\documentclass[a4paper, reqno, 11pt]{amsart}

\usepackage[english]{babel}
\usepackage{amsmath}
\usepackage{amssymb}
\usepackage{enumerate}
\usepackage{ifthen}
\usepackage{bbm}
\usepackage{color}
\RequirePackage{amsmath}
\RequirePackage{amssymb}
\RequirePackage{mathrsfs}
\usepackage{tikz}
\usepackage{verbatim}

\provideboolean{shownotes} 
\setboolean{shownotes}{true}
\usepackage{hyperref}

\newcommand{\margnote}[1]{
\ifthenelse{\boolean{shownotes}}%
{\marginpar{\raggedright\tiny\texttt{#1}}}%
{}%
}

\newcommand{\hole}[1]{
\ifthenelse{\boolean{shownotes}}%
{\begin{center} \fbox{ \rule {.25cm}{0cm}
\rule[-.1cm]{0cm}{.4cm} \parbox{.85\textwidth}{\begin{center}
\texttt{#1}\end{center}} \rule {.25cm}{0cm}}\end{center}}
{}
}
\RequirePackage{amsthm}
\theoremstyle{plain}
\newtheorem{definition}{Definition}

\newtheorem{theorem}{Theorem}

\newtheorem{lemma}{Lemma}

\newtheorem{assumption}{Assumption}

\theoremstyle{remark}
\newtheorem{remark}{Remark}

\newcommand{\eps}{\varepsilon}

\newcommand{\NN}{\mathbb{N}}

\newcommand{\RR}{\mathbb{R}}

\newcommand{\cc}[1]{\mathcal{#1}}

\renewcommand{\ss}[1]{\mathscr{#1}}

\newcommand{\wave}[1]{\widetilde{#1}}
            
\newcommand{\setof}[2]{\left\{\left.\vphantom{#2} #1 \;\right|\,\; #2 \right\}}
\renewcommand{\and}{\; , \;}

\newcommand{\conv}{\mathop{\mathrm{conv}}\nolimits}

\newcommand{\Le}{\mathscr{L}}

\newcommand{\weakto}{\rightharpoonup}

\newcommand{\Linf}{L^\infty}

\newcommand{\dist}{\mathop{\mathrm{dist}}\nolimits}

\newcommand{\inte}{\mathop{\mathrm{int}}\nolimits}
\renewcommand{\d}{\partial}
            
\renewcommand{\div}{\mathop{\mathrm{div}}\nolimits}

\newcommand{\rb}[1]{\left( #1 \right)}



\begin{document}

\title[Non-uniqueness and prescribed energy]{Non-uniqueness and prescribed energy \\ for the continuity equation}

\author[Crippa]{Gianluca Crippa}
\address[G. Crippa]{
Departement Mathematik und Informatik, Universit\"at
  Basel\\ Rheinsprung 21 \\CH-4051 Basel \\Switzerland}
\email[]{\href{gianluca.crippa@unibas.ch}{gianluca.crippa@unibas.ch}}

\author[Gusev]{Nikolay Gusev}
\address[N. Gusev]{
Dybenko st., 22/3, 94 \\ 125475 Moscow \\ Russia}
\email[]{\href{n.a.gusev@gmail.com}{n.a.gusev@gmail.com}}

\author[Spirito]{Stefano Spirito}
\address[S. Spirito]{
GSSI - Gran Sasso Science Institute\\ Viale Francesco Crispi 7\\ 67100 L'Aquila\\ Italy}
\email[]{\href{stefano.spirito@gssi.infn.it}{stefano.spirito@gssi.infn.it}}

\author[Wiedemann]{Emil Wiedemann}
\address[E. Wiedemann]{
Hausdorff Center for Mathematics and Mathematical Institute, Universit\"at Bonn\\
Endenicher Allee 60\\53115 Bonn\\Germany}
\email[]{\href{emil.wiedemann@hcm.uni-bonn.de}{emil.wiedemann@hcm.uni-bonn.de}}

\begin{abstract}
In this note we provide new non-uniqueness examples for the continuity equation
by constructing infinitely many weak solutions with prescribed energy.
\end{abstract}

\maketitle

\section{Introduction}
In this paper we consider the \emph{continuity equation} for a bounded scalar function $u\colon \RR \times \RR^d \to \RR$ with a bounded divergence-free vector field $\bold b \colon \RR \times \RR^d \to \RR^d$:
\begin{gather}
\d_t u + \div (u \bold b) = 0, \label{continuity} \\
\div \bold b = 0. \label{div-free}
\end{gather}
This equation appears in various problems of mathematical physics, in particular fluid mechanics and kinetic theory.
In the smooth setting (and assuming suitable integrability) the \emph{energy}
\begin{equation*}
{\mathscr E}(t) := \int_{\RR^d} u^2(t,x)\, dx
\end{equation*}
of the solution $u$ is conserved:
\begin{equation}\label{conservation}
\frac{d}{dt} {\mathscr E} (t) = 0.
\end{equation}
Indeed, since $\bold b$ is divergence-free, by multiplying \eqref{continuity} with $u$, using the chain rule and integrating over $\RR^d$ one immediately obtains \eqref{conservation}.

In many applications one has to study \eqref{continuity} in a nonsmooth setting. Roughly speaking, since \eqref{continuity} is linear,
the conservation of energy \eqref{conservation} 
implies uniqueness of weak solutions to the corresponding initial-value problem for \eqref{continuity}. In fact, conservation of energy is a consequence of the so-called \emph{renormalization property}, which was proved by \cite{DPL} for any vector field $\bold b$ with Sobolev regularity and later extended by Ambrosio \cite{A2004} to the case when $b$ has bounded variation. We refer to \cite{DLB,HW} for a detailed review of recent results in this direction.

On the other hand, when the regularity of the vector field $\bold b$ is too low, the conservation of energy~\eqref{conservation} fails in general. In a nonsmooth setting several counterexamples to the uniqueness, and therefore to the conservation of energy, are known, see \cite{A, CLR2003, Depauw, ABC1, ABC2}. A similar phenomenon occurs in the context of the Euler equations. For example, in the papers \cite{Sch,Shn,DlSz2008DI} weak solutions of the Euler equations were constructed with compact support in space time.

In particular the example in \cite{Depauw} gives
a bounded vector field $\bold b$ and a bounded scalar field $u$ which satisfy \eqref{continuity} and \eqref{div-free} such that
\begin{equation}\label{defect}
{\mathscr E}(t) = \left\{ \begin{array}{ll}
0 & \text{ for $t \leq 0$} \\
1 & \text{ for $t>0$.}
\end{array}\right.
\end{equation}

In this paper, for any given nonnegative bounded function $E\colon \RR \to \RR$ which is continuous on an open interval and zero outside we construct infinitely many pairs $(\bold b, u)$ satisfying \eqref{continuity} and \eqref{div-free}, such that ${\mathscr E}(t) = E(t)$
for a.e.~$t$. Thus, in contrast with~\eqref{defect}, we provide more general profiles for the energy.
%
%
%
%
%
%
%
%
Our results are also connected to the chain rule problem for the divergence operator, see \cite{ADM,BG,CGSWRD}.

We construct such pairs $(\bold b , u)$ using the method of convex integration, and our techniques are similar to the ones used in \cite{DlSz2008DI,Sz2012RPM}. The latter reference contains an appendix giving a general framework for convex integration, but
for the problem at hand we need to consider a nonlinear constraint that depends on the points in the domain (as was the case e.g.~in \cite{DlSz2010AC}, albeit in a different functional setting). For this reason we adapt the framework from \cite{Sz2012RPM} to this more general situation (see \S\ref{framework}). We then apply this abstract framework to the specific situation of the continuity equation (see \S\ref{application}).

Finally, let us mention \cite{Cordoba2011,Shvydkoy,Lopes}, where results were obtained by convex integration, respectively, that yield as a byproduct counterexamples to the energy conservation for continuity equations. However, in these works the energy profile is always piecewise constant. 

\subsection*{Acknowledgements} This research has been partially supported by the SNSF grants 140232 and 156112. This work was started while the third author was a PostDoc at the Departement Mathematik und Informatik of the Universit\"at Basel. He would like to thank the department for the hospitality and the support. The second author was partially supported by the Russian Foundation for Basic Research, project no.~13-01-12460. The authors are grateful to S.~Bianchini, C.~De Lellis, and L.~Sz\'{e}kelyhidi for the fruitful discussions about the topic of the paper.

\section{Differential inclusions with non-constant nonlinear constraint}\label{framework}
We start with the so-called Tartar framework (cf.~e.g.~\cite{DlSz2008DI}).
Consider a system of $m$ linear partial differential equations
\begin{equation}\label{L}
\sum_{i=1}^D A_i \d_i z = 0
\end{equation}
in an open set $\ss D\subset \RR^D$, where $A_i$ are constant $m\times n$ matrices and $z\colon \ss D \to \RR^n$. Consider a nonlinear constraint
\begin{equation}\label{N}
z(y)\in K_y \quad
\end{equation}
for a.e. $y$ in $\ss D$,
where $K_y\subset \RR^n$ is a compact set for any $y\in \ss D$.


For any $y\in \ss D$ let $U_y:= \inte \conv K_y$, where with $\conv$ we denote the convex hull of the set $K_y$ and with  $\inte$ we denote its interior. 
Let $\ss U \subset \ss D$ be a bounded open set.

\begin{definition}[Subsolutions]
We say that $z\in L^2(\ss D)$ is a \emph{subsolution} of \eqref{L}, \eqref{N} if $z$ is a weak solution of \eqref{L} in $\ss D$, $z$ is continuous on $\ss U$, \eqref{N} holds for a.e. $y\in \ss D \setminus \ss U$ and
\begin{equation}\label{N-relaxed}
z(y) \in U_y
\end{equation}
for any $y\in \ss U$.
\end{definition}
\begin{definition}[Localized plane waves/wave cone]\label{def:wave cone}
A set $\Lambda  \subset \RR^n$ is called wave cone if there exists a constant $C>0$ such that for any $\bar z \in \Lambda$ there exists a sequence $w_k\in C_0^\infty(B_1(0);\RR^n)$ solving \eqref{L} in $\RR^D$ such that
\begin{itemize}
\item $\dist(w_k(x),[-\bar z, \bar z])\to 0$ for all $x\in B_1(0)$ uniformly as $k\to \infty$,
\item $w_k \weakto 0$ in $L^2$ as $k\to \infty$,
\item $\int |w_k|^2 \, dy > C |\bar z|^2$ for all $k\in \NN$.
\end{itemize}
\end{definition}
In the above definition we denoted the segment with endpoints $x\in \RR^n$ and $y\in \RR^n$ with $[x,y]:=\conv\{x,y\}$.
The functions $w_k$ are called \emph{localized plane waves}.
We make the following assumptions:
\begin{assumption}\label{a:wc}
There exists a wave cone $\Lambda$ dense in $\RR^n$.
\end{assumption}
Let $\ss K$ denote the set of all compact subsets of $\RR^n$, endowed with the Hausdorff metric $d_\cc{H}$. It is well-known that $\ss K$ is a complete metric space.
\begin{assumption}[Continuity of the nonlinear constraint]\label{a:Ky}
The map $f\colon \ss U \ni y \mapsto K_y \in \ss K$ is continuous and bounded in the Hausdorff metric.
\end{assumption}

Our main abstract result is the following:

\begin{theorem}\label{t:ci}
Suppose that Assumptions~\ref{a:wc} and \ref{a:Ky} hold. Suppose that $z_0$ is a subsolution of \eqref{L}, \eqref{N}. Then there exist infinitely many weak solutions $z\in L^2(\ss D)$ of \eqref{L} which agree with $z_0$ a.e. on $\ss D \setminus \ss U$ and satisfy \eqref{N} for a.e. $y\in \ss D$.
\end{theorem}

\subsection{Geometric preliminaries}
The next lemma shows that compact subsets of the interior of the convex hull of a compact set $K$ are stable with respect to sufficiently small perturbations of $K$ in the Hausdorff metric.

\begin{lemma}\label{l:stab}
Let $K\subset \RR^n$ be a compact set. Then for any compact set $C\subset \inte \conv K$ there exists $\eps>0$ such that for any compact set  $K' \subset \RR^n$ with $d_{\cc H}(K,K')<\eps$ we have
\begin{equation*}
C\subset \inte \conv K' .
\end{equation*}
\end{lemma}

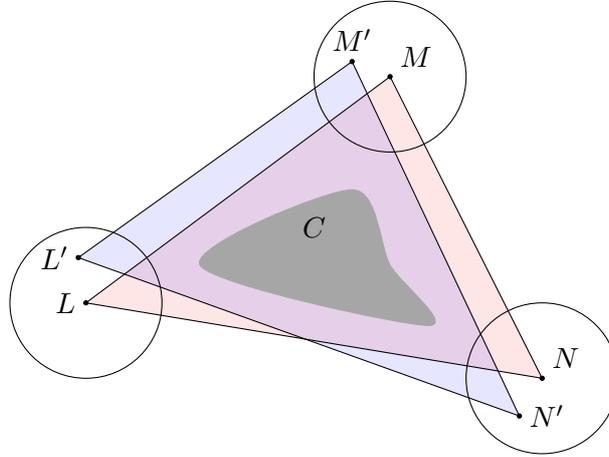
\begin{figure}[h!]
\caption{An illustration of Lemma~\ref{l:stab} in the case when $K=\{L,M,N\}$ and $K'=\{L',M',N'\}$.}
\begin{tikzpicture}
\coordinate(A) at (0,0) {};
\coordinate(B) at (4,3) {};
\coordinate(C) at (6,-1) {};
\coordinate(A') at (-0.1,0.6) {};
\coordinate(B') at (3.5,3.2) {};
\coordinate(C') at (5.7,-1.5) {};
\node[left] at (A){$L$};
\node[above right] at (B){$M$};
\node[above right] at (C){$N$};
\node[left] at (A'){$L'$};
\node[above] at (B'){$M'$};
\node[right] at (C'){$N'$};
\fill[red,opacity=0.1] (A) -- (B) -- (C) -- cycle;
\fill[blue,opacity=0.1] (A') -- (B') -- (C') -- cycle;
\draw (A)--(B)--(C)--cycle;
\draw (A')--(B')--(C')--cycle;
\filldraw [gray!70!white] plot [smooth cycle] coordinates {(1.5,0.5) (3.5,1.5) (4,0.5) (4.5,-0.3)};
\node at (3,1){$C$};
\draw (A) circle (1cm);
\draw (B) circle (1cm);
\draw (C) circle (1cm);
\fill (A) circle (1pt);
\fill (B) circle (1pt);
\fill (C) circle (1pt);
\fill (A') circle (1pt);
\fill (B') circle (1pt);
\fill (C') circle (1pt);
\end{tikzpicture}
\end{figure}

\begin{proof}[Proof.]
Since $\inte \conv K$ is open,
for any point $x\in C$ there exists a simplex $S_x$ with vertices $\{v_i\}_{i=1..n+1}\subset \conv K$ such that $x$ belongs to the inner open simplex 
\begin{equation*}
I_x:=\setof{\sum_{i=1}^{n+1} \lambda_i v_i}{\lambda_i\in \rb{\frac{1}{2(n+1)},\frac{2}{n+1}}, \; \sum_{i=1}^{n+1} \lambda_i=1, \; i=1..n+1}.
\end{equation*}
Since $C$ is a compact set and the inner simplices $\{I_x\}_{x\in C}$ cover $C$ we can extract a finite subcover $\{I_{x_k}\}_{k=1..m}$ of $C$. 
Let us fix $k\in 1..m$ and consider the simplex $S:=S_{x_k}$ with vertices $\{v_i\}_{i=1..n+1}\subset \conv K$. Let $I:=I_{x_k}$ denote the corresponding inner simplex.

If $\eps<\dist(\d I, \d S)$ then for any points $v_i'\in B_\eps(v_i)$, $i=1..n+1$ one has
\begin{equation}\label{iss}
I\subset \conv \{v_1', v_2', \dots, v_{n+1}'\}.
\end{equation}
Observe that for any $\eps>0$ and $i=1..n+1$ the ball $B_\eps(v_i)$ contains a point $v_i'\in\conv K'$.
Indeed, by Caratheodory's theorem $v_i= \sum_{j=1}^{n+1} \mu_j z_j$ for some $z_j \in K$ and $\mu_j \in [0,1]$ with $\sum_{j=1}^{n+1}\mu_j=1$. Since $d_\cc{H}(K,K')<\eps$ there exist points $z_j'\in K'$ such that $z_j'\in B_\eps(z_j)$, where $j=1..n+1$. Let 
\begin{equation*}
v_i':=\sum_{j=1}^{n+1} \mu_j z_j',
\end{equation*}
then $|v_i - v_i'| \le \sum_{j=1}^{n+1} \mu_j |z_j - z_j'| < \eps$. Hence by \eqref{iss} we have $I\subset \conv \{v_1', v_2', \dots, v_{n+1}'\}$ provided that $\eps$ is small enough. But $v_i'\in\conv K'$, hence $I \subset \conv K'$. Since $I$ is open we can also write $I \subset \inte \conv K'$.

Since we have finitely many simplices, we can choose $\eps>0$ in such a way that the inclusion $I_{x_k}\subset \inte \conv K'$ holds for any $k=1..m$ (provided that $d_\cc{H}(K,K')<\eps$). Then
\begin{equation*}
C \subset \cup_{k=1..m} I_{x_k} \subset \inte \conv K'. \qedhere
\end{equation*}
\end{proof}


We will also need the following elementary lemma:

\begin{lemma}\label{l:cont}
Suppose that $z\in C(\ss U; \RR^n)$ where $\ss U\subset \RR^D$ is an open set. Suppose that for any $y\in \ss U$ we have a compact set $K_y\subset \RR^n$ and the function $y\mapsto K_y$ is continuous in the Hausdorff metric. Then the function $F\colon y \mapsto \dist(z(y),K_y)$ is continuous on $\ss U$.
\end{lemma}
\begin{proof}
One can prove directly that the function $(z,K)\mapsto \dist(z,K)$ is continuous on $\RR^n\times \ss K$. The function $y\mapsto (z(y), K_y)$ is continuous in view of the assumptions. Hence the function $F$ is continuous as a composition of continuous functions.
\end{proof}

In general the distance from a point $z$ to a compact set $K$ does not control from below the distance from $z$ to the boundary of $\conv K$. However the following lemma shows that there exists a segment inside $\inte \conv K$ with midpoint $z$ and length controlled from below by $\dist(z,K)$:
\begin{lemma}[Geometric lemma]\label{l:geom}
Let $K\subset \RR^n$ be a compact set. For any $z\in \inte \conv K$ there exists $\bar z\in \RR^n$ such that 
\begin{itemize}
\item $[z-\bar z, z+ \bar z] \subset \inte \conv K$
\item $|\bar z| \ge \frac{1}{2n} \dist(z,K)$ 
\end{itemize}
\end{lemma}
(This is exactly Lemma 5.3 from \cite{DLSz2012h}.)

\subsection{Convex integration}
The following lemma is the main building block of the convex integration scheme:
\begin{lemma}[Perturbation lemma]\label{l:perturb}
Suppose that Assumptions~\ref{a:wc} and \ref{a:Ky} hold and $z$ is a subsolution of \eqref{L} and \eqref{N} such that
\begin{equation*}
\int_\ss{U} \dist^2(z(y),K_y) \, dy = \varepsilon>0.
\end{equation*}
Then there exists $\delta=\delta(\varepsilon)>0$ and a sequence $\{z_k\}_{k\in\NN}$ of subsolutions of \eqref{L} and \eqref{N} such that
\begin{itemize}
\item $z_k = z$ on $\ss D \setminus \ss U$ for any $k\in \NN$
\item $\int_\ss{U} |z-z_k|^2 \, dy \ge \delta$ for any $k\in \NN$
\item $z_k \weakto z$ in $L^2(\ss U)$ as $k\to \infty$.
\end{itemize}
\end{lemma}
\begin{proof}
\emph{Step 1.}
Let $y\in \ss U$. Since $z(y)\in U_y$ we can apply Lemma~\ref{l:geom} to obtain $\bar z_*(y)$ such that
\begin{gather*}
[z(y)-\bar z_*(y), z(y) + \bar z_*(y)] \subset U_y, \\
|\bar z_*(y)| \ge \frac{1}{2n} \dist(z(y),K_y),
\end{gather*}
Since $\Lambda$ is dense in $\RR^n$ and $U_y$ is open we can find $\bar z(y) \in \Lambda$ such that
\begin{gather}
[z(y)-\bar z(y), z(y) + \bar z(y)] \subset U_y, \label{seg}\\
|\bar z(y)| \ge \frac{1}{4n} \dist(z(y),K_y). \label{amplitude}
\end{gather}

Due to \eqref{seg} there exists $\rho(y)>0$
\begin{equation*}
[z(y)-\bar z(y), z(y) + \bar z(y)]+\overline{B_{2\rho(y)}(0)} \subset U_y.
\end{equation*}
Hence using Assumption~\ref{a:Ky}, Lemma~\ref{l:stab} and the  continuity of $z$ we can find $R(y)>0$ such that
\begin{equation}\label{stability}
[z(x)-\bar z(y), z(x) + \bar z(y)]+\overline{B_{\rho(y)}(0)} \subset U_x
\end{equation}
for all $x\in B_{R(y)}(y)\subset \ss U$.
Moreover, in view of Lemma~\ref{l:cont} we can choose $R(y)$ in such a way that
\begin{equation}\label{stability-dist}
\dist(z(x), K_x) \le 2 \dist(z(y),K_y)
\end{equation}
for all $x\in B_{R(y)}(y)$.

Using Assumption~\ref{a:wc} for any fixed $y\in \ss U$ we can construct a sequence $\{w_{y,k}\}_{k\in \NN} \subset C_0^\infty(B_1(0))$ such that
\begin{itemize}
\item $w_{y,k}(x) \in [-\bar z(y), \bar z(y)]+B_{\rho(y)}(0)$ for all $x\in B_1(0)$ and $k\in \NN$,
\item $w_{y,k} \weakto 0$ in $L^2$ as $k\to \infty$,
\item $\int |w_{y,k}|^2 \, dx > C |\bar z(y)|^2$ for all $k\in \NN$.
\end{itemize}

\emph{Step 2.}
Let $\eps:=\int_{\ss U} \dist^2 (z(y), K_y) \, dy$.
The balls $\setof{B_r(y)}{y\in \ss U, \; r\in (0,R(y))}$ cover $\ss U$, so using Vitali's covering theorem
(see e.g. \cite{Bogachev}, Theorem~5.5.2)
and the absolute continuity of the Lebesgue integral we can find finitely many points $\{y_i\}_{i=1..N}\subset \ss U$ and radii $r_i\in(0,R(y_i))$ such that
\begin{equation}\label{exhaustion}
\sum_{i=1}^{N} \int_{B_i} \dist^2(z(y), K_y) \, dy > \frac{1}{2} \eps,
\end{equation}
where the balls $B_i:= B_{r_i}(y_i)$ are pairwise disjoint.

For each $i=1..N$ let us introduce the scaled and translated perturbations $w_{i,k}(x):=w_{y_i,k}(\frac{x-y_i}{r_i})$. These functions belong to $C_0^\infty(B_i)$ and satisfy
\begin{enumerate}
\item[(i)] $w_{i,k}(x) \in [-\bar z(y_i), \bar z(y_i)]+B_{\rho(y_i)}(0)$ for all $x\in B_i$,
$k\in \NN$, $i= 1..N$;
\item[(ii)] $w_{i,k} \weakto 0$ in $L^2$ as $k\to \infty$ (for each fixed $i= 1..N$);
\item[(iii)] $\int |w_{i,k}|^2 \, dx > C |\bar z(y_i)|^2 \Le^D(B_i)$ for all $k\in \NN$.
\end{enumerate}

In view of (i) and \eqref{stability} we have $z(x)+ w_{i,k}(x)\in U_x$ for all $x\in \ss U$ and $i=1..N$, hence $z+w_{i,k}\in X_0$.
Since the balls $B_i$ are pairwise disjoint the function 
\begin{equation*}
z_k:= z+ \sum_{i=1}^N w_{i,k}
\end{equation*}
also belongs to $X_0$.

Using successively (iii), \eqref{amplitude}, \eqref{stability-dist} and \eqref{exhaustion}
we obtain:
\begin{multline*}
\int_{\ss U} |z - z_k|^2 \, dy = \sum_{i=1}^N \int_{B_i} |w_{i,k}(y)|^2 \, dy
\stackrel{\text{(iii)}}{>} C\sum_{i=1}^N |\bar z(y_i)|^2 \Le^D(B_i)  \\
\stackrel{\eqref{amplitude}}{\ge} \frac{C}{16n^2}\sum_{i=1}^N \dist^2(z(y_i), K_{y_i}) \Le^D(B_i) =
\frac{C }{16n^2}\sum_{i=1}^N \int_{B_i} \dist^2(z(y_i), K_{y_i}) \, dx \\
\stackrel{\eqref{stability-dist}}{>}  \frac{C}{32n^2} \sum_{i=1}^N \int_{B_i} \dist^2(z(x), K_x ) \, dx \stackrel{\eqref{exhaustion}}{>} \frac{C}{64n^2} \eps.
\end{multline*}
It remains to observe that since $N$ is finite and the points $y_i$ are fixed we have $z_k \weakto z$ in $L^2$ as $k\to \infty$.
\end{proof}

\subsection{Proof of Theorem~\ref{t:ci}}

We are now ready to prove our main abstract theorem.

\begin{proof}[Proof of Theorem~\ref{t:ci}]
Let $X_0$ denote a set of all subsolutions of \eqref{L} and \eqref{N} which agree with $z_0$ on $\ss D \setminus \ss U$. 
Let $X$ be the closure of $X_0$ in the weak topology of $L^2(\ss U)$, 
endowed with the corresponding induced weak topology. Clearly any $z\in X$ solves \eqref{L} and satisfies \eqref{a:Ky} a.e. on $\ss D \setminus \ss U$.


For any $z\in X$ let us define
\begin{equation*}
I(z):= \int_\ss{U} |z(y)|^2 \, dy.
\end{equation*}
This functional is a Baire-1 function on $X$. Indeed, for any $j\in \NN$ let
\begin{equation*}
I_j(z) := \int_{\setof{y\in \ss U}{\dist(y,\d U)>1/j}} |(\omega_{1/j} * z)(y)|^2 \, dy
\end{equation*}
where for any $\eps>0$ we denote by $\omega_\eps(\cdot) = \eps^{-D} \omega(\cdot/\eps)$ the standard convolution kernel. Then for any $j\in \NN$ the functional $I_j$ is continuous on $X$, and for any $z\in X$ we have $I_j(z) \to I(z)$ as $j\to \infty$.

In view of Assumption~\ref{a:Ky} $X$ is a \emph{bounded} subset of $L^2(\ss U)$.
Since the weak topology is metrizable on the norm-bounded subsets of $L^2(\ss U)$,
we can consider $X$ as a complete metric space with some metric $d_X$.

Then by Baire category theorem (see also Theorem~7.3 from \cite{Oxt}) the set
\begin{equation*}
Y:=\setof{z\in X}{I \text{ is continuous at } z}
\end{equation*}
is residual in $X$ (and hence is infinite). We claim that $z\in Y$ implies $J(z)=0$,
where
\begin{equation*}
J(z):= \int_\ss{U} \dist^2(z(y),K_y) \, dy.
\end{equation*}

Indeed, suppose that $J(z)=\eps>0$ for some $z\in Y$. Let $z_j \in X_0$ be a sequence such that $z_j \weakto z$ in $L^2(\ss U)$ as $j\to \infty$. Since $I$ is continuous at $z$ this implies that $I(z_j)\to I(z)$ and consequently $z_j \to z$ in $L^2(\ss U)$ as $j\to \infty$.

Then in view of Assumption~\ref{a:Ky} we have $J(z_j) \to J(z)$ as $j\to \infty$ and hence without loss of generality we can assume that $J(z_j) > \eps/2$ for all $j\in \NN$.

Applying Lemma~\ref{l:perturb} to $z_j$ for each $j\in \NN$ we can find $\wave z_j \in X_0$ such that $d_X(\wave z_j,z_j)<2^{-j}$ and $\int_{\ss U} |\wave z_j - z_j|^2 \, dy \ge \delta > 0$,
where $\delta=\delta(\eps)$ is independent of $j$.

Since $d_X(\wave z_j, z) \le d_X(\wave z_j, z_j) + d_X(z_j, z) \to 0$ as $j\to \infty$ we also have $\wave z_j \weakto z$ in $L^2$. Since $z$ is a point of continuity of $I$ we also have $z_j \to z$ in $L^2(\ss U)$ as $j\to \infty$. But then $\wave z_j - z_j \to 0$ in $L^2(\ss U)$, which contradicts the construction of $\wave z_j$.
\end{proof}

\section{Application to the continuity equation}\label{application}

In this section we apply our abstract framework to the case of the continuity equation.

\begin{theorem} \label{t:antirenorm}
Suppose that $d \ge 2$. Let $E\colon \RR \to \RR$ be a non-negative bounded function which is continuous on some bounded open interval $I\subset \RR$ and vanishes on $\RR \setminus I$. Then there exist infinitely many bounded, compactly supported $u\colon \RR \times \RR^d \to \RR$ and $\bold b \colon \RR \times \RR^d \to \RR^d$ which satisfy \eqref{continuity} and \eqref{div-free} in sense of distributions and such that 
\begin{equation*}
\int_{\RR^2} u^2(t,x) \, dx =  E(t) \qquad \textrm{for a.e. }t\in I.
\end{equation*}
\end{theorem}
\begin{remark}
It is well-known that a representative of $u$ can be chosen such that the map $t \mapsto u(t,\cdot)$ is continuous with values in $L^2$ equipped with the weak topology. Then the question arises whether the assertion in the theorem holds even for \emph{every}, and not just almost every, time $t$. We expect this to be true: indeed this should follow by methods similar to those of~\cite{DlSz2010AC}. We will however not pursue this question further in this article. 
\end{remark}
\begin{remark}
When $d=2$ and $f$ is a characteristic function of an interval, the statement of Theorem~\ref{t:antirenorm}, essentially, follows from the example constructed in \cite{Depauw}. This particular case of Theorem~\ref{t:antirenorm} was also proved in \cite{Gus2011} using the convex integration method.
\end{remark}


\begin{remark}
A similar problem can be addressed for more general equation of the form $\div (u \bold B)=0$ instead of \eqref{continuity}.
For this equation the problem is stated as follows: given a distribution $g$ is it possible to construct compactly supported bounded functions $u\colon \RR^n \to \RR$, $\bold B\colon \RR^n \to \RR^n$ such that
\begin{equation*}
\div ( u \bold B ) = 0, \quad \div \bold B = 0, \quad \div (u^2 \bold B) = g \; ?
\end{equation*}
This is related to the so-called \emph{chain rule problem} for the divergence operator \cite{ADM}.
When $n=2$ such a construction is not possible for $g\ne 0$ in view of \cite{BG},
while for $n\ge 3$ it is possible and is obtained using convex integration and rank-2 laminates in \cite{CGSWRD}.
\end{remark}

Let us put the continuity equation in the framework of the previous section.
Fix a bounded open set $\Omega \subset \RR^d$.
Let $\ss U:= I\times\Omega$ and 
\[
F(t,x) := \frac{E(t)}{\Le^d(\Omega)} \mathbf{1}_{\Omega}(x),
\]
where $\mathbf{1}_{\Omega}$ denotes the characteristic function of $\Omega$.

We consider equations \eqref{continuity} and \eqref{div-free} as a linear system
\begin{gather}
\d_t u + \div_x \bold m = 0, \label{L-continuity} \\
\div_x \bold b = 0 \label{L-div-free}
\end{gather}
in $\ss D:= \RR \times \RR^d$
with $u\colon \ss D \to \RR$, $\bold m \colon \ss D \to \RR^d$ and $\bold b \colon \ss D \to \RR^d$
such that $z:=(u, \bold m, \bold b)$ satisfies the constraint
\begin{equation}
z(y) \in K_y := 
\begin{cases}
\setof{(u, \bold m, \bold b)}{\bold m = u \bold b, \quad
|\bold b| = 1, \quad u^2 = F(y)}
& \text{if \;\;} y\in \ss U
\\
0
& \text{if \;\;} y\in \ss D \setminus \ss U
\end{cases}
\label{N-continuity}
\end{equation}
for a.e. $y = (x,t) \in \ss D$.

Suppose that $z=(u,\bold m, \bold b)\in \Linf(\ss D)$ satisfies \eqref{L-continuity} and \eqref{L-div-free} in sense of distributions and moreover \eqref{N-continuity} holds a.e.~in $\ss D$. Then the couple $(u, \bold b)$ satisfies the assertion of Theorem~\ref{t:antirenorm}. 


\medskip

Let us check the assumption of Theorem~\ref{t:antirenorm}.

\begin{lemma} \label{l:dH-est}
Suppose that $A,B\subset \RR^n$ are compact sets and $r>0$ is such that
\begin{itemize}
\item for any $z \in A$ there exists $z' \in B \cap B_r(z)$
\item for any $z \in B$ there exists $z' \in A \cap B_r(z)$
\end{itemize}
Then $d_\cc{H}(A,B) < r$.
\end{lemma}

\begin{proof}
Suppose that $d_\cc{H}(A,B) \ge r$. Then without loss of generality we can assume that there exists $z \in A$ such that for any $z'\in B$ we have $z \notin B_r(z')$. But then the ball $B_r(z)$ cannot contain any point of $B$, which leads to a contradiction.
\end{proof}

\begin{lemma} \label{l:Ky-cont}
If $F \colon \ss U\to \RR$ is continuous, bounded and non-negative then the map
$y \mapsto K_y$ is continuous and bounded (w.r.t. $d_\cc{H}$) on $\ss U$.
\end{lemma}

\begin{proof}
Let $f(y):=\sqrt{F(y)}$.
Let us fix $y\in \ss U$.
For any $\eps >0$ let $\delta >0 $ be such that $|f(y)-f(y')|<\eps$ for any $y'\in B_\delta(y) \subset \ss U$.
Let us prove that $d_\cc{H}(K_{y}, K_{y'})< 2\eps $ for all $y'\in B_\delta(y)$.

For any $z \in K_y$ there exist $\sigma \in \{\pm 1\}$ and $\bold b\in \RR^d$ with $|\bold b|=1$ such that $z=(\sigma f(y), \sigma f(y) \bold b, \bold b)$.
Then $z':=(\sigma f(y'), \sigma f(y') \bold b, \bold b)$ belongs to $K_{y'}$
and $|z-z'|\le 2|f(y)-f(y')|$. Hence there exists $z' \in K_{y'} \cap B_{2\eps}(z)$.

Similarly, for any $z' \in K_{y'}$ there exist $\sigma \in \{\pm 1\}$ and $\bold b\in \RR^d$ with $|\bold b|=1$ such that $z'=(\sigma f(y'), \sigma f(y') \bold b, \bold b)$.
Then $z:=(\sigma f(y), \sigma f(y) \bold b, \bold b)$ belongs to $K_{y}$
and $|z-z'|\le 2|f(y)-f(y')|$. Hence there exists $z \in K_{y} \cap B_{2\eps}(z')$.

Therefore by Lemma~\ref{l:dH-est} we have $d_\cc{H}(K_{y},K_{y'})< 2\eps$.
\end{proof}

\begin{lemma} \label{l:wc}
Assumption \ref{a:wc} holds for the system \eqref{L-continuity}--\eqref{N-continuity}.
\end{lemma}
\begin{proof}
Let $\phi \colon \ss D \to \RR$ be a non-negative smooth function such that $0 \le \phi \le 1$ on $\ss D$,
$\phi = 0$ on $\ss D \setminus B_1(0)$ and $\phi = 1$ on $B_{1/2}(0)$.

\emph{Part 1.} Suppose that $d > 2$. Let us show that Assumption \ref{a:wc} holds with $\Lambda = \RR^{2d+1}$.
Fix $\bar u \in \RR$, $\bar {\bold m} \in \RR^d$ and $\bar {\bold b} \in \RR^d$ and let $\bar z = (\bar u, \bar {\bold m}, \bar {\bold b})$.
Since $d > 2$ there exists a unit vector $\bold n \in \RR^d$ such that $\bold n \cdot \bar {\bold m} = \bold n \cdot \bar {\bold b} = 0$.
Denote $\hat {\bold n} = (0, \bold n)$, $\bar {\bold a} = (\bar u, \bar {\bold m})$.
For any $k \in \NN$ define $\bar {\bold a}_k \colon \ss D \to \RR^{d+1}$ by
\begin{equation*}
\bar {\bold a}_k(y):= \bar {\bold a} (\hat {\bold n} \cdot \nabla_{y} (\phi \Pi_k))
 - \hat {\bold n} (\bar {\bold a} \cdot \nabla_{y}(\phi \Pi_k))
\end{equation*}
where $y=(t,x)$ and
\begin{equation*}
\Pi_k(y) := \frac{\sin(k \hat {\bold n} \cdot y)}{k}.
\end{equation*}
Observe that
\begin{equation*}
\div_{y} \bar {\bold a}_k = (\bar {\bold a} \cdot \nabla_{y}) (\hat {\bold n} \cdot \nabla_{y}) (\phi \Pi_k)
 - (\hat {\bold n} \cdot \nabla_{y}) (\bar {\bold a} \cdot \nabla_{y})(\phi \Pi_k) = 0.
\end{equation*}
Let $(u_k, \bold m_k)$ denote the components of $\bar {\bold a}_k$, then by the equation above we have $\d_t u_k + \div_x \bold m_k = 0$.

Similarly let
\begin{equation*}
{\bold b}_k(t,x) := \bar {\bold b} ({\bold n} \cdot \nabla_x (\phi \Pi_k))
 - {\bold n} (\bar {\bold b} \cdot \nabla_x(\phi \Pi_k))
\end{equation*}
Then arguing as above $\div \bold b_k = 0$.

Now we introduce $w_k := (u_k, \bold m_k, \bold b_k)$. Then
\begin{equation*}
w_k(y) = \bar z \phi \cos (k \hat {\bold n} \cdot y) + f \Pi_k
\end{equation*}
where $f$ does not depend on $k$ and vanishes on $B_{1/2}(0)$.

On the other hand,
\begin{multline} \label{eq:plane-wave-energy}
\int_{\ss D} |w_k|^2 \, dy \ge \int_{B_{1/2}(0)} |w_k|^2 \, dy
=
\int_{B_{1/2}(0)} |\bar z|^2 \cos^2 (k \hat {\bold n} \cdot y) \, dy
= \\ =
\int_{B_{1/2}(0)} |\bar z|^2 \frac{1 + \cos (2 k \hat {\bold n} \cdot y)}{2} \, dy
\ge
\frac{|\bar z|^2}{4}|B_{1/2}(0)|
\end{multline}
provided that $k$ is sufficiently large.

\emph{Part 2.} Suppose that $d=2$ and fix $\bar z=(\bar u, \bar {\bold m}, \bar {\bold b})\in \RR\times\RR^2\times\RR^2$ with $\bar u\not=0$. Let us look for a localized plane wave in the following form:
\begin{equation*}
w_k=(\bold a_k, \bold b_k)
\end{equation*}
with
\begin{equation*}
\begin{aligned}
&\bold a_k(y)=\nabla_{y}\times\left(\phi \bold A\frac{\sin(k  {\bold n}\cdot y)}{k}\right)\\
&\bold b_k(t,x)=\nabla_{x}^{\perp}\left(\phi \frac{\sin\bigl( k {\bold n}\cdot (t,x) \bigr)}{k}\right)
\end{aligned}
\end{equation*}
where $ {\bold n}=( n_t, {\bold n}_x)\in\RR\times\RR^2$ and $\bold A\in\RR^3$ are to be chosen and $k\in\NN$. Then, by construction 
\begin{equation*}
\div_{y} \bold a_k=0, \qquad \div_{x} \bold b_k=0.
\end{equation*}
Then, we get
\begin{equation*}
w_k =\hat{z}\phi \cos (k  {\bold n} \cdot y) + f\frac{\sin(k {\bold n}\cdot y)}{k} 
\end{equation*}
where $\hat{z}=(\bold A\times {\bold n}, {\bold n}_x^{\perp})$ and $f$ does not depend on $k$ and vanishes on $B_{1/2}(0)$.

In order to have $\hat z = \bar z$ the vectors $\bold A$ and ${\bold n}$ must satisfy
\begin{equation*}
\begin{aligned}
\bold A\times {\bold n}&=(\bar u, \bar {\bold m}),\\
{\bold n}_x^{\perp}&=\bar {\bold b}.
\end{aligned}
\end{equation*}
From the second equation we immediately obtain that ${\bold n}_x=-\bar{\bold b}^{\perp}$. 
Since $\bar u\not= 0$ there exists $n_t$ such that ${\bold n}\perp (\bar u,\bar {\bold m})$. Then, 
we can always find $\bold A$ such that the first equation is satisfied.
It remains to observe that the estimate \eqref{eq:plane-wave-energy} holds also in the considered case.
We thus have verified Assumption~\ref{a:wc} for $\Lambda = \RR^5 \setminus \{ \bar u = 0 \}$.
\end{proof}

\begin{proof}[Proof of Theorem~\ref{t:antirenorm}]
By symmetry of $K_y$ for any $y\in \ss U$ we have $0 \in \inte \conv K_y$. On the other hand $K_y = \{0\}$ for any $y \in \ss D \setminus \ss U$. Therefore $u\equiv 0$, $\bold m \equiv 0$ and $\bold b \equiv 0$ is a subsolution of \eqref{L-continuity}--\eqref{N-continuity}.
Then the result follows from Lemma~\ref{l:cont}, Lemma~\ref{l:wc} and Theorem~\ref{t:ci}.
\end{proof}

\bibliography{energy}

\providecommand{\bysame}{\leavevmode\hbox to3em{\hrulefill}\thinspace}
\providecommand{\MR}{\relax\ifhmode\unskip\space\fi MR }
\providecommand{\MRhref}[2]{%
  \href{http://www.ams.org/mathscinet-getitem?mr=#1}{#2}
}
\providecommand{\href}[2]{#2}
\begin{thebibliography}{ADLM07}

\bibitem[ABC13]{ABC2}
G.~Alberti, S.~Bianchini, and G.~Crippa, \emph{Structure of level sets and
  {S}ard-type properties of {L}ipschitz maps: results and counterexamples},
  Ann. Scuola Norm. Sup. Pisa Cl. Sci. \textbf{5} (2013), 863--902.

\bibitem[ABC14]{ABC1}
\bysame, \emph{A uniqueness result for the continuity equation in two
  dimensions}, J. Eur. Math. Soc. (JEMS) \textbf{16} (2014), 201--234.

\bibitem[AC14]{HW}
L.~Ambrosio and G.~Crippa, \emph{Continuity equations and {ODE} flows with
  non-smooth velocity}, Proc. Roy. Soc. Edinburgh Sect. A \textbf{144} (2014),
  no.~6, 1191--1244.

\bibitem[ADLM07]{ADM}
L.~Ambrosio, C.~De~Lellis, and J.~Mal{\'y}, \emph{On the chain rule for the
  divergence of {BV}-like vector fields: applications, partial results, open
  problems}, Perspectives in nonlinear partial differential equations, Contemp.
  Math., vol. 446, Amer. Math. Soc., Providence, RI, 2007, pp.~31--67.

\bibitem[Aiz78]{A}
M.~Aizenman, \emph{On vector fields as generators of flows: A counterexample to
  {N}elson's conjecture}, Ann. Math. \textbf{107} (1978), 287--296.

\bibitem[Amb]{A2004}
L.~Ambrosio, \emph{Transport equation and {C}auchy problem for {${B}{V}$}
  vector fields}, Invent. Math.

\bibitem[BG14]{BG}
S.~Bianchini and N.~A. Gusev, \emph{Steady nearly incompressible vector fields
  in 2d: chain rule and renormalization}, Preprint (2014).

\bibitem[BLFNL]{Lopes}
A.~Bronzi, M.~Lopes~Filho, and H.~Nussenzveig~Lopes, \emph{Wild solutions for
  2d incompressible ideal flow with passive tracer}, Preprint.

\bibitem[Bog07]{Bogachev}
V.~I. Bogachev, \emph{Measure theory. {V}ol. {I}, {II}}, Springer-Verlag,
  Berlin, 2007.

\bibitem[CFG11]{Cordoba2011}
D.~Cordoba, D.~Faraco, and F.~Gancedo, \emph{Lack of uniqueness for weak
  solutions of the incompressible porous media equation}, Arch. Ration. Mech.
  Anal. \textbf{200} (2011), no.~3, 725--746.

\bibitem[CGSW]{CGSWRD}
G.~Crippa, N.~A. Gusev, S.~Spirito, and E.~Wiedemann, \emph{Failure of the
  chain rule for the divergence of bounded vector fields}, Preprint.

\bibitem[CLR03]{CLR2003}
F.~Colombini, T.~Luo, and J.~Rauch, \emph{Uniqueness and nonuniqueness for
  nonsmooth divergence free transport}, Seminaire: \'{E}quations aux
  {D}\'eriv\'ees {P}artielles, 2002--2003, S\'emin. \'Equ. D\'eriv. Partielles,
  Exp.\ No.\ XXII, \'Ecole Polytech., Palaiseau, 2003, pp.~1--21.

\bibitem[Dep03]{Depauw}
N.~Depauw, \emph{Non unicit\'e des solutions born\'ees pour un champ de
  vecteurs {BV} en dehors d'un hyperplan}, C. R. Math. Acad. Sci. Paris
  \textbf{337} (2003), no.~4, 249--252.

\bibitem[DL89]{DPL}
R.~J. DiPerna and P.-L. Lions, \emph{Ordinary differential equations, transport
  theory and {S}obolev spaces}, Invent. Math. \textbf{98} (1989), no.~3,
  511--547.

\bibitem[DL08]{DLB}
C.~De~Lellis, \emph{Ordinary differential equations with rough coefficients and
  the renormalization theorem of {A}mbrosio [after {A}mbrosio, {D}i{P}erna,
  {L}ions]}, Ast\'erisque (2008), no.~317, Exp. No. 972, viii, 175--203,
  S{\'e}minaire Bourbaki. Vol. 2006/2007.

\bibitem[DLS09]{DlSz2008DI}
C.~De~Lellis and L.~Sz{\'e}kelyhidi, Jr., \emph{The {E}uler equations as a
  differential inclusion}, Ann. of Math. \textbf{170} (2009), no.~3,
  1417--1436.

\bibitem[DLS10]{DlSz2010AC}
\bysame, \emph{On admissibility criteria for weak solutions of the {E}uler
  equations}, Arch. Ration. Mech. Anal. \textbf{195} (2010), no.~1, 225--260.

\bibitem[DLS12]{DLSz2012h}
\bysame, \emph{The {$h$}-principle and the equations of fluid dynamics}, Bull.
  Amer. Math. Soc. (N.S.) \textbf{49} (2012), no.~3, 347--375.

\bibitem[Gus11]{Gus2011}
N.~A. Gusev, \emph{Continuity equation as a differential inclusion},
  Unpublished (2011).

\bibitem[Oxt80]{Oxt}
J.~Oxtoby, \emph{Measure and category}, second ed., Graduate Texts in
  Mathematics, vol.~2, Springer-Verlag, New York-Berlin, 1980, A survey of the
  analogies between topological and measure spaces.

\bibitem[Sch93]{Sch}
V.~Scheffer, \emph{An inviscid flow with compact support in space-time}, J.
  Geom. Anal. \textbf{3} (1993), no.~4, 343--401.

\bibitem[Shn97]{Shn}
A.~Shnirelman, \emph{On the nonuniqueness of weak solution of the {E}uler
  equation}, Comm. Pure Appl. Math. \textbf{50} (1997), no.~12, 1261--1286.

\bibitem[Shv11]{Shvydkoy}
R.~Shvydkoy, \emph{Convex integration for a class of active scalar equations},
  J. Amer. Math. Soc. \textbf{24} (2011), no.~4, 1159--1174. \MR{2813340
  (2012d:35295)}

\bibitem[Sz{\'e}12]{Sz2012RPM}
L.~Sz{\'e}kelyhidi, Jr., \emph{Relaxation of the incompressible porous media
  equation}, Ann. Sci. \'Ec. Norm. Sup\'er. (4) \textbf{45} (2012), no.~3,
  491--509.

\end{thebibliography}
\bibliographystyle{amsalpha}

\end{document}